\documentclass{amsart}
\usepackage{graphicx}
\usepackage[top=1.5in,bottom=1.5in,left=1.5in,right=1.5in]{geometry}
\usepackage{leftidx}
\usepackage{amssymb}
\newtheorem{theorem}{Theorem}[section]
\newtheorem*{theorem A}{Theorem A}
\newtheorem*{theorem B}{N\"olker's Theorem}
\newtheorem{lemma}{Lemma}[section]

\newtheorem{proposition}{Proposition}[section]

\newtheorem{corollary}{Corollary}[section]

\theoremstyle{remark}
\newtheorem{remark}{Remark}[section]
\theoremstyle{remark}

\theoremstyle{definition}

\numberwithin{equation}{section}
\def\({\left ( }
\def\){\right )}
\def\<{\left < }
\def\>{\right >}

\begin{document}

\title{ Quasi-Einstein hypersurfaces of complex space forms  }

\author{Xiaomin Chen}
\address{College of  Science, China University of Petroleum-Beijing, Beijing, 102249, China}
\email{xmchen@cup.edu.cn}

\thanks{The author is supported by Natural Science Foundation of Beijing, China (Grant No.1194025).}


\begin{abstract}
Based on a well-known fact that there are no Einstein hypersurfaces in a non-flat complex space form, in this article we study the quasi-Einstein condition,  which is a generalization of an Einstein metric, on the real hyersurface of a non-flat complex space form. For the real hypersurface with quasi-Einstein metric of a complex Euclidean space, we also give a classification. Since a gradient Ricci soliton is a special quasi-Einstein metric, our results improve some conclusions of \cite{CK}.
\end{abstract}


\maketitle

\section{Introduction}

Denote by $\widetilde{M}^n$  the complex space form, i.e. a complex $n$-dimensional K\"ahler manifold with constant holomorphic sectional
curvature $c$. A complete and simple connected complex space form  is complex analytically isometric to a complex projective space
 $\mathbb{C}P^n$ if $c>0$, a complex hyperbolic space $\mathbb{C}H^n$ if $c<0$, a complex Euclidean space $\mathbb{C}^n$ if $c=0.$
The complex projective and complex hyperbolic spaces are called \emph{non-flat complex space forms} and denoted by $\widetilde{M}^n(c)$.
Let $M$ be a real hypersurface of a complex space form. In particular, if $\xi$ is an eigenvector of shape operator $A$ then $M$ is
called a \emph{Hopf hypersurface}. Since there are no
Einstein real hypersurfaces in $\widetilde{M}^n(c)$ (\cite{CR,M}), a natural question is whether there is a generalization of an Einstein metric in the real hyersurface of $\widetilde{M}^n(c)$.    A Ricci soliton is a Riemannian metric, which satisfies
\begin{equation*}
 \frac{1}{2}\mathcal{L}_V g+Ric-\lambda g=0,
\end{equation*}
where $V$ and $\lambda$ are the potential vector field and some constant, respectively. It is clear that a trivial Ricci soliton is an Einstein metric with $V$ zero or Killing. When the potential vector field $V$ is a gradient vector field, i.e. $V=\nabla f$, where $f$ is a smooth function, then it is called a \emph{gradient Ricci soliton.}
 Cho and Kimura \cite{CK,CK2} proved that a Hopf hypersurface and a non-Hopf hypersurface in a non-flat complex space form do not admit a gradient Ricci soliton. Moreover, this is true when the gradient Ricci soliton is repalced by a compact Ricci soliton due to Perelman's result (\cite[Remark 3.2]{P}).

As another interesting generalization of an Einstein metric, a quasi-Einstein metric has been considered (see \cite{CSW,C}). We call a triple $(M, g, f, m)$ (a Riemannian manifold $(M, g)$ with a function
$f$ on $M$) ($m$-)\emph{quasi-Einstein} if it satisfies the equation
\begin{equation}\label{1}
{\rm Ric}+{\rm Hess} f-\frac{1}{m}df\otimes df = \lambda g
\end{equation}
for some $\lambda\in\mathbb{R}$, where $m$ is a positive integer. ${\rm Hess} f$ denotes the Hessian of $f$.
 Notice that Equation \eqref{1} recovers the gradient Ricci soliton when $m=\infty$. A quasi-Einstein metric is an Einstein metric if $f$ is constant. We call a quasi-Einstein
metric \emph{shrinking, steady or expanding}, respectively, when $\lambda<0, \lambda=0$ or $\lambda>0$.
For a general manifold,  quasi-Einstein metrics have been studied in depth and some rigid properties and gap results were obtained (cf.\cite{CSW,W,W2}).
 On the other hand, we also notice that for the odd-dimensional manifold, Ghosh in \cite{G} studied quasi-Einstein contact metric manifolds.
As is well known that a real hypersurface of $\widetilde{M}^n(c)$ is a $(2n-1)$-dimensional almost contact manifold and a gradient Ricci soliton is just a special quasi-Einstein metric with $m=\infty$. From this observation we are inspired to improve the results of \cite{CK} and study the quasi-Einstein condition for the real hypersurface of a complex space form.

In this article, we first study the quasi-Einstein metric on Hopf hypersurfaces in complex space forms as well as a class of non-Hopf hypersurfaces in non-flat complex space forms.
\begin{theorem}\label{T1}
There are no quasi-Einstein Hopf real hypersurfaces in a non-flat complex space form.
\end{theorem}
\begin{theorem}\label{T2}
There are no quasi-Einstein ruled hypersurfaces  in a non-flat complex space form.
\end{theorem}
\begin{remark}
Since a gradient Ricci soliton is a special quasi-Einstein metric with $m=\infty$, Theorem \ref{T1} and Theorem \ref{T2} improve the results of \cite{CK}.
\end{remark}

Also we consider the real hypersurfaces with a quasi-Einstein metric of complex Euclidean space $\mathbb{C}^n$ as in \cite{CK}.
We first suppose that $M$ is a contact hypersurface of complex Euclidean space $\mathbb{C}^n$, i.e. $\phi A+A\phi=2\sigma\phi$, where $\sigma>0$ is a smooth function.
\begin{theorem}\label{T3}
Let $M^{2n-1}$ be a complete contact hypersurface of complex Euclidean space $\mathbb{C}^n$. If $M$ admits a quasi-Einstein metric, then $M$ is  a sphere $\mathbb{S}^{2n-1}$ or a generalized cylinder $\mathbb{R}^{n}\times\mathbb{S}^{n-1}$.
\end{theorem}
For a general hypersurface of complex Euclidean space $\mathbb{C}^n$, we obtain
\begin{corollary}\label{T4}
Let $M^{2n-1}$ be a complete real hypersurface with $A\xi=0$ of complex Euclidean space $\mathbb{C}^n$. If $M$ admits a non-steady quasi-Einstein metric,  it is a hypersphere, hyperplane or developable hypersurface.
\end{corollary}

 In order to prove these conclusions,  we need recall some basic concepts and related results in Section 2. In Section 3 and Section 4, we give respectively the proofs of Theorem \ref{T1} and Theorem \ref{T2}, and the real hypersurface with a quasi-Einstein metric of complex Euclidean spaces is presented in Section 5.

\section{Some basic concepts and related results}

  Let ($\widetilde{M}^n,\widetilde{g})$ be a complex $n$-dimensional K\"ahler manifold
   and $M$ be an immersed, without boundary, real hypersurface of $\widetilde{M}^n$ with the induced metric $g$.
Denote by $J$ the complex structure on $\widetilde{M}^n$. There exists a local defined
unit normal vector field $N$ on $M$ and we write $\xi:=-JN$
by the structure vector field of $M$.
 An induced one-form $\eta$ is defined by
$\eta(\cdot)=\widetilde{g}(J\cdot,N)$, which is dual to $\xi$.  For any vector field $X$ on $M$ the tangent part of $JX$
is denoted by $\phi X=JX-\eta(X)N$. Moreover, the following identities hold:
\begin{equation}\label{2}
\phi^2=-Id+\eta\otimes\xi,\quad\eta\circ \phi=0,\quad\phi\circ\xi=0,\quad\eta(\xi)=1,
\end{equation}
\begin{equation}\label{3}
g(\phi X,\phi Y)=g(X,Y)-\eta(X)\eta(Y),
\end{equation}
\begin{equation}\label{4}
g(X,\xi)=\eta(X),
\end{equation}
where $X,Y\in\mathfrak{X}(M)$. By \eqref{2}-\eqref{4}, we know that $(\phi,\eta,\xi,g)$ is an almost
contact metric structure on $M$.

Denote by $\nabla, A$ the induced Riemannian connection and the shape operator on $M$, respectively.
Then the Gauss and Weingarten formulas are given by
\begin{equation}\label{5}
\widetilde{\nabla}_XY=\nabla_XY+g(AX,Y)N,\quad\widetilde{\nabla}_XN=-AX,
\end{equation}
 where $\widetilde{\nabla}$ is the connection on $\widetilde{M}^n$ with respect to $\widetilde{g}$.
Also, we have
\begin{equation}\label{6}
  (\nabla_X\phi)Y=\eta(Y)AX-g(AX,Y)\xi,\quad\nabla_X\xi=\phi AX.
\end{equation}
In particular, $M$ is said to be a \emph{Hopf hypersurface} if the structure vector field $\xi$ is an eigenvector of $A$, i.e. $A\xi=\alpha\xi$, where $\alpha=\eta(A\xi).$

From now on we always assume that the holomorphic sectional curvature of $\widetilde{M}^n$ is constant $c$. When $c=0$, $\widetilde{M}^n$ is complex Euclidean space $\mathbb{C}^n$. When $c\neq0$, $\widetilde{M}^n$ is a non-flat complex space form, denoted by $\widetilde{M}^n(c)$, then from \eqref{5}, we know that the curvature tensor $R$ of $M$ is given by
\begin{align}\label{7}
R(X,Y)Z=\frac{c}{4}\Big(&g(Y,Z)X-g(X,Z)Y+g(\phi Y,Z)\phi X-g(\phi X,Z)\phi Y\\
&+2g(X,\phi Y)\phi Z\Big)+g(AY,Z)AX-g(AX,Z)AY\nonumber
\end{align}
and the shape operator $A$ satisfies
\begin{equation}\label{8}
 (\nabla_XA)Y-(\nabla_YA)X=\frac{c}{4}\Big(\eta(X)\phi Y-\eta(Y)\phi X-2g(\phi X,Y)\xi\Big)
\end{equation}
for any vector fields $X,Y,Z$ on $M$.
From \eqref{7}, we get for the Ricci tensor $Q$ of type $(1,1)$:
\begin{equation}\label{9}
 QX=\frac{c}{4}\{(2n+1)X-3\eta(X)\xi\}+hAX-A^2X,
\end{equation}
where $h$ denotes the mean curvature of $M$ (i.e. $h={\rm trace}(A)$). We denote $S$ the scalar curvature of $M$, i.e. $S={\rm trace}(Q).$

Now we suppose $M$ is an Hopf hypersurface. Differentiating $A\xi=\alpha\xi$ covariantly gives
\begin{equation}\label{9*}
  (\nabla_XA)\xi=X(\alpha)\xi+\alpha\phi AX-A\phi AX.
\end{equation}
 Using  \eqref{8}, we obtain
\begin{equation}\label{10}
 (\nabla_\xi A)X=X(\alpha)\xi+\alpha\phi AX-A\phi AX+\frac{c}{4}\phi X
\end{equation}
for any vector field $X$. Since $\nabla_\xi A$ is self-adjoint, by taking the anti-symmetry part of \eqref{10}, we get the relation:
\begin{equation}\label{11}
2A\phi AX-\frac{c}{2}\phi X=X(\alpha)\xi-\eta(X)\nabla\alpha+\alpha(\phi A+A\phi)X.
\end{equation}

As the tangent bundle $TM$ can be decomposed as $TM=\mathbb{R}\xi\oplus\mathfrak{D}$, where $\mathfrak{D}=\{X\in TM:X\bot\xi\}$, the condition $A\xi=\alpha\xi$ implies $A\mathfrak{D}\subset\mathfrak{D}$, thus we can pick up $X\in\mathfrak{D}$ such that $AX=\mu X$ for some function $\mu$ on $M$.
Then from \eqref{11} we obtain
\begin{equation}\label{12}
(2\mu-\alpha)A\phi X=\Big(\mu\alpha+\frac{c}{2}\Big)\phi X.
\end{equation}
If $2\mu=\alpha$ then $c=-4\mu^2,$ which show that $M$ is locally congruent to a horosphere in $\mathbb{C}H^n$(see \cite{B}).

Next we recall two important lemmas for a Riemannian manifold satisfying quasi-Einstein equation \eqref{1}.
\begin{lemma}[\cite{G}]\label{L1}
For a quasi-Einstein metric, the curvature tensor $R$ can be expressed as
\begin{align*}
R(X,Y )\nabla f =&(\nabla_YQ)X-(\nabla_XQ)Y-\frac{\lambda}{m}\{X(f)Y-Y(f)X\}\\
&+\frac{1}{m}\{X(f)QY-Y(f)QX\}
\end{align*}
for any vector fields $X,Y$ on $M$.
\end{lemma}
\begin{lemma}[\cite{CSW}]
For a quasi-Einstein $(M^{2n-1}, g, f, m)$, the following equations hold:
\begin{align}
\frac{1}{2}\nabla S =& \frac{m-1}{m}Q(\nabla f)+\frac{1}{m}\Big(S-(2n-2)\lambda\Big)\nabla f\label{2.13},\\
  \frac{1}{2}\Delta S-\frac{m+2}{2m}g(\nabla f,\nabla S)=& -\frac{m-1}{m}\Big|{\rm Ric}-\frac{S}{2n-1}g\Big|^2\label{3.10*}\\
  &-\frac{m+2n-2}{m(2n-1)}\Big(S-(2n-1)\lambda\Big)\Big(S-\frac{(2n-2)(2n-1)}{m+2n-2}\lambda\Big).\nonumber
\end{align}
\end{lemma}
Applying Lemma \ref{L1} we obtain
\begin{lemma}
For a quasi-Einstein Hopf real hypersurface $M^{2n-1}$ of a complex space form $\widetilde{M}^n$, the following equation holds:
\begin{align}\label{13}
  \alpha(\phi A^2 +A^2\phi)=(\alpha^2+c)(A\phi+\phi A)+(h-\frac{\alpha}{2})c\phi.
\end{align}
\end{lemma}
\begin{proof}
 Replacing $Z$ in \eqref{7} by $\nabla f$, we have
 \begin{align*}
R(X,Y)\nabla f=\frac{c}{4}\Big(&Y(f) X-X(f)Y+\phi Y(f)\phi X-\phi X(f)\phi Y\\
&+2g(X,\phi Y)\phi\nabla f\Big)+AY(f)AX-AX(f)AY.\nonumber
\end{align*}
By  Lemma \ref{L1}, we get
\begin{align}\label{15}
  &(\nabla_YQ)X-(\nabla_XQ)Y+\frac{1}{m}\{X(f)QY-Y(f)QX\} \\
 = &\Big(\frac{c}{4}-\frac{\lambda}{m}\Big)\Big(Y(f) X-X(f)Y\Big)+\frac{c}{4}\Big(\phi Y(f)\phi X-\phi X(f)\phi Y\nonumber\\
&+2g(X,\phi Y)\phi\nabla f\Big)+AY(f)AX-AX(f)AY.\nonumber
\end{align}

Now making use of \eqref{9}, for any vector fields $X,Y$ we first compute
\begin{align*}
  (\nabla_YQ)X=&\frac{c}{4}\{-3(\nabla_Y\eta)(X)\xi-3\eta(X)\nabla_Y\xi\}+Y(h)AX+h(\nabla_YA)X\\
  &-(\nabla_YA)AX-A(\nabla_YA)X\\
  =&-\frac{3c}{4}\{g(\phi AY,X)\xi+\eta(X)\phi AY\}+Y(h)AX+h(\nabla_YA)X\\
  &-(\nabla_YA)AX-A(\nabla_YA)X.
\end{align*}
By \eqref{8}, we thus obtain
\begin{align}\label{16}
  &(\nabla_XQ)Y-(\nabla_Y Q)X \\
   =&-\frac{3c}{4}\{g(\phi AX+A\phi X,Y)\xi+\eta(Y)\phi AX-\eta(X)\phi AY\}\nonumber\\
   &+X(h)AY-Y(h)AX+\frac{hc}{4}\Big(\eta(X)\phi Y-\eta(Y)\phi X-2g(\phi X,Y)\xi\Big)\nonumber\\
   &-(\nabla_XA)AY+(\nabla_YA)AX-\frac{c}{4}\Big(\eta(X)A\phi Y-\eta(Y)A\phi X-2g(\phi X,Y)A\xi\Big).\nonumber
\end{align}

Since $M$ is Hopf, i.e. $A\xi=\alpha\xi$, taking the product of \eqref{15} with $\xi$ and using \eqref{16}, we conclude that
\begin{align}\label{17}
  &-\frac{1}{m}\{X(f)\eta(QY)-Y(f)\eta(QX)\}+\Big(\frac{c}{4}-\frac{\lambda}{m}\Big)\Big(Y(f)\eta(X)-X(f)\eta(Y)\Big)\\
  &+\alpha\Big(AY(f)\eta(X)-AX(f)\eta(Y)\Big)-\frac{3c}{4}g(\phi AX+A\phi X,Y)\nonumber\\
   &+\alpha\Big(X(h)\eta(Y)-Y(h)\eta(X)\Big)-\frac{h-\alpha}{2}cg(\phi X,Y)\nonumber\\
   &-g((\nabla_XA)AY-(\nabla_YA)AX,\xi)=0.\nonumber
\end{align}
 Moreover, using \eqref{9*} we compute
\begin{align*}
  &g((\nabla_XA)AY-(\nabla_YA)AX,\xi) \\
   =&g(X(\alpha)\xi+\alpha\phi AX-A\phi A X,AY) -g(Y(\alpha)\xi+\alpha\phi AY-A\phi AY,AX)\\
   =&\alpha[X(\alpha)\eta(Y)-Y(\alpha)\eta(X)]+2\alpha g(\phi AX,AY)-g(A\phi AX,AY)+g(A\phi AY,AX).
\end{align*}
Substituting this into \eqref{17} and using \eqref{9}, we arrive at
\begin{align*}
 &\Big(-\frac{1}{m}[\frac{c}{2}(n-1)+h\alpha-\alpha^2]-\frac{c}{4}+\frac{\lambda}{m}\Big)\{X(f)\eta(Y)-Y(f)\eta(X)\}\\
  &+\alpha\Big(AY(f)\eta(X)-AX(f)\eta(Y)\Big)-\frac{3c}{4}g(\phi AX+A\phi X,Y)\nonumber\\
   &+\alpha\Big(X(h)\eta(Y)-Y(h)\eta(X)\Big)-\frac{h-\alpha}{2}cg(\phi X,Y)\nonumber\\
   &-\alpha[X(\alpha)\eta(Y)-Y(\alpha)\eta(X)]-2\alpha g(\phi AX,AY)\\
   &+g(A\phi AX,AY)-g(A\phi AY,AX)=0.\nonumber
\end{align*}

Moreover, applying \eqref{11} in the above formula we have
\begin{align}\label{18}
  &\Big(-\frac{1}{m}[\frac{c}{2}(n-1)+h\alpha-\alpha^2]-\frac{c}{4}+\frac{\lambda}{m}\Big)\{X(f)\eta(Y)-Y(f)\eta(X)\}\\
  &+\alpha\Big(AY(f)\eta(X)-AX(f)\eta(Y)\Big)-\frac{c}{2}g(\phi AX+A\phi X,Y)\nonumber\\
   &+\alpha\Big(X(h)\eta(Y)-Y(h)\eta(X)\Big)-\frac{h-\alpha}{2}cg(\phi X,Y)\nonumber\\
   &-\frac{\alpha}{2}[X(\alpha)\eta(Y)-Y(\alpha)\eta(X)]-\alpha g(\phi AX,AY)\nonumber\\
   &+g(-\frac{1}{2}\eta(X)\nabla\alpha+\frac{1}{2}\alpha(A\phi)X,AY)\nonumber\\
   &-g(-\frac{1}{2}\eta(Y)\nabla\alpha+\frac{1}{2}\alpha(A\phi)Y,AX)=0.\nonumber
\end{align}
Replacing $X$ and $Y$ by $\phi X$ and $\phi Y$ respectively and using \eqref{11} again yields \eqref{13}.
\end{proof}

\section{Proof of Theorem \ref{T1}}
In this section we assume $c\neq0$. Let $M^{2n-1}$ be a Hopf hypersurface of $\widetilde{M}^n(c)$, i.e. $A\xi=\alpha\xi$, then $\alpha$ is constant due to \cite[Theorem 2.1]{NR}. We first consider $\alpha=0$, i.e. $A\xi=0$, then Equation \eqref{13} implies
\begin{equation}\label{19}
  A\phi +\phi A +h\phi =0.
\end{equation}

Let $X\in\mathfrak{D}$ be a principle vector field  corresponding to principle curvature $\mu$, then from \eqref{19} we know that $\phi X$ is also
a principle vector field with principle curvature $(-h-\mu)$. Thus we see that the mean curvature $h$ must be zero, i.e. $A\phi+\phi A=0$, which implies $c=0$ by the result of \cite{KS}. Hence we obtain the following:
\begin{proposition}\label{P1}
An Hopf hypersurface of $\widetilde{M}^n(c),c\neq0$ with $A\xi=0$  does not admit a quasi-Einstein metric.
\end{proposition}
Next we consider the case where $\alpha\neq0$.
If $A$ has only one principle curvature $\frac{\alpha}{2}$ in $\mathfrak{D}$, the mean curvature $h=n\alpha$ is constant. From \eqref{18} we can obtain
\begin{align*}
  &\Big(-\frac{1}{m}\Big[\frac{c}{2}(n-1)+h\alpha-\alpha^2\Big]-\frac{c}{4}+\frac{\lambda}{m}\Big)\{X(f)\eta(Y)-Y(f)\eta(X)\}\\
  &+\alpha\Big(AY(f)\eta(X)-AX(f)\eta(Y)\Big)-\frac{c}{2}g(\phi AX+A\phi X,Y)\nonumber\\
   &-\frac{n-1}{2}\alpha cg(\phi X,Y)=0.\nonumber
\end{align*}
Letting $X\in\mathfrak{D}$  such that $AX=\frac{\alpha}{2}X$ and taking $Y=\phi X$, we arrive at $nc=0$. It is impossible.

Now choose $X\in\mathfrak{D}$ such that $AX=\mu X$ with $\mu\neq\frac{\alpha}{2}$, so from \eqref{13} we have
\begin{align}\label{28}
  \alpha(\mu^2+\widetilde{\mu}^2)=(\alpha^2+c)(\mu+\widetilde{\mu})+(h-\frac{\alpha}{2})c.
\end{align}
Here we have used $A\phi X=\widetilde{\mu}\phi X$ with $\widetilde{\mu}=\frac{\mu\alpha+\frac{c}{2}}{2\mu-\alpha}$ followed from \eqref{12}.

Moreover, inserting $\widetilde{\mu}=\frac{\mu\alpha+\frac{c}{2}}{2\mu-\alpha}$ into the equation \eqref{28}, we have
\begin{align}\label{29}
  &4\alpha \mu^4-4(c+2\alpha^2)\mu^3+(4\alpha c+4\alpha^3-4hc)\mu^2\\
  &+(4hc\alpha-2\alpha^2c-c^2)\mu+\frac{3}{4}\alpha c^2+\alpha^3c-hc\alpha^2=0.\nonumber
\end{align}

Now we denote the roots of the polynomial by $f_1,f_2,f_3,f_4$, then from the relation between the roots and coefficients we obtain
\begin{align}\label{30}
\left\{
  \begin{array}{ll}
    &f_1+f_2+f_3+f_4=\frac{c+2\alpha^2}{\alpha },  \\
    &f_1f_2+f_1f_3+f_1f_4+f_2f_3+f_2f_4+f_3f_4 =\frac{\alpha c+\alpha^3-hc}{\alpha},  \\
    &f_1f_2f_3+f_1f_2f_4+f_2f_3f_4 =-\frac{4hc\alpha-2\alpha^2c-c^2}{4\alpha},  \\
    &f_1f_2f_3f_4 =\frac{3c^2+4\alpha^2c-4hc\alpha}{16}.
  \end{array}
\right.
\end{align}
As the proof of \cite[Lemma 4.2]{CK}, we can also get the following lemma.
\begin{lemma}\label{L2}
 The mean curvature $h$ is constant.
\end{lemma}
Hence taking $Y=\xi$ in \eqref{18} we conclude
\begin{align}\label{3.10}
  &\theta\{\nabla f-\xi(f)\xi\}+\alpha\Big(\alpha\xi(f)\xi-A\nabla f\Big)=0.
\end{align}
where
\begin{equation*}
\theta:=-\frac{1}{m}\Big[\frac{c}{2}(n-1)+h\alpha-\alpha^2\Big]-\frac{c}{4}+\frac{\lambda}{m}.
\end{equation*}
By taking the inner product of \eqref{3.10} with the principal vector $X\in\mathfrak{D}$, we obtain
\begin{equation*}
(\alpha\mu-\theta)X(f)=0.
\end{equation*}

If $\alpha\mu-\theta\neq0$, then $\nabla f=\xi(f)\xi.$ Differentiating this along any vector field $Z$ gives
\begin{equation}\label{21}
  \nabla_Z\nabla f=Z(\xi(f))\xi+\xi(f)\phi AZ.
\end{equation}
Since $d^2f=0$, i.e. $g(\nabla_Z\nabla f,W)=g(\nabla_W\nabla f,Z)$ for any vector fields $Z,W$, it follows from \eqref{21} that
\begin{equation*}
  g(Z(\xi(f))\xi+\xi(f)\phi AZ,W)=g(W(\xi(f))\xi+\xi(f)\phi AW,Z).
\end{equation*}
Replacing $Z$ and $W$ by $\phi Z$ and $\phi W$ respectively implies
\begin{equation}\label{3.7*}
 \xi(f)(\phi A Z+A\phi Z)=0.
\end{equation}
This implies $\xi(f)=0$ since $\phi A+A\phi =0$ will yield $c=0$ (\cite{KS}). Thus $f$ is constant and $M$ is Einstein, which is impossible. So $\alpha\mu-\theta=0$, i.e. $M$ has at most two distinct constant principal curvatures $\alpha,\mu=\frac{\theta}{\alpha}$.
This shows that the scalar curvature $S$ is constant.

Using \eqref{3.10} we derive from \eqref{9} that
\begin{align*}
 Q(\nabla f)=&\frac{c}{4}\{(2n+1)\nabla f-3\xi(f)\xi\}+hA\nabla f-A^2\nabla f\\
 =&\frac{c}{4}\{(2n+1)\nabla f-3\xi(f)\xi\}+h(\alpha-\mu)\xi(f)\xi+h\mu\nabla f\nonumber\\
 &-A\Big((\alpha-\mu)\xi(f)\xi+\mu\nabla f\Big)\nonumber\\
 =&\Big(\frac{c}{4}(2n+1)+h\mu-\mu^2\Big)\nabla f-\Big(\alpha^2-\mu^2+\frac{3c}{4}-h(\alpha-\mu)\Big)\xi(f)\xi.
\end{align*}

If $m\neq1$, by \eqref{2.13} we have
\begin{align*}
 &\Big(\alpha^2-\mu^2+\frac{3c}{4}-h(\alpha-\mu)\Big)\xi(f)\xi\\
 =&\Big(\frac{c}{4}(2n+1)+h\mu-\mu^2+\frac{1}{m-1}(S-(2n-2)\lambda)\Big)\nabla f,\nonumber
\end{align*}
which, by taking the inner product with any vector field $X\in\mathfrak{D}$, yields
\begin{align}\label{3.8}
  \frac{c}{4}(2n+1)+h\mu-\mu^2+\frac{1}{m-1}(S-(2n-2)\lambda)=0.
\end{align}
Here we have used $g(X,\nabla f)\neq0$ for some vector field $X\in\mathfrak{D}$. Otherwise, if $g(X,\nabla f)=0$ for all $X\in\mathfrak{D}$, then
$f$ is constant since $\xi(f)=0$, which is impossible as before.

 Since the hypersurface $M$ has two distinct constant principle curvatures: $\alpha$ of multiplicity $1$ and $\mu$ of multiplicity $2n-2$,
it is easy to get that the mean curvature $h=\alpha+(2n-2)\mu$  and the scalar curvature $S=c(n^2-1)+2\alpha(2n-2)\mu+(2n-2)(2n-3)\mu^2.$

Furthermore, since $A$ has only one eigenvalue $\mu=\frac{\theta}{\alpha}$ in $\mathfrak{D}$, we see from \eqref{12} that
\begin{equation}\label{31}
2\mu^2-2\alpha \mu-\frac{c}{2}=0.
\end{equation}
 By \eqref{31}, the scalar curvature $S$ may be written as
\begin{equation}\label{3.10**}
  S=(n-1)\Big(\frac{c}{2}(4n-1)+2(2n-1)\alpha\mu\Big).
\end{equation}
Using \eqref{31} again and $h=\alpha+(2n-2)\mu$, we thus have
\begin{align*}
  \theta=&-\frac{1}{m}[\frac{c}{2}(n-1)+h\alpha-\alpha^2]-\frac{c}{4}+\frac{\lambda}{m}\\
=&-\frac{n-1}{m}\Big(\frac{c}{2}+2\mu\alpha\Big)-\frac{c}{4}+\frac{\lambda}{m}.
\end{align*}
Since $\mu-\frac{\theta}{\alpha}=0$, we obtain
\begin{equation}\label{32}
  \lambda=(m+2n-2)\mu\alpha+\frac{n-1}{2}c+\frac{mc}{4}.
\end{equation}
Inserting \eqref{3.10**} and \eqref{32} into \eqref{3.8}, we derive from \eqref{31}
\begin{align*}
0=&\frac{c}{4}(4n-2)+(2n-2)\alpha\mu+\frac{n-1}{m-1}\Big(\frac{c}{2}(4n-1)+2(2n-1)\alpha\mu-2\lambda\Big)\\
=&\frac{c}{2}(2n-1)+(2n-2)\alpha\mu+\frac{n-1}{m-1}\Big(\frac{c}{2}(2n+1)-2(m-1)\mu\alpha-\frac{mc}{2}\Big)\\
  =&\frac{nc}{2}\Big(1+\frac{2n-2}{m-1}\Big),
\end{align*}
which leads to $nc=0.$ The contradiction implies $m=1$.

Since the scalar curvature is constant, by \eqref{2.13} we get $S=(2n-2)\lambda$.  Because \eqref{3.10**} and \eqref{32} still hold for $m=1$, if $S=(2n-2)\lambda$ we obtain
\begin{equation*}
  (n-1)\Big(\frac{c}{2}(4n-1)+2(2n-1)\alpha\mu\Big)=(2n-2)\Big((2n-1)\mu\alpha+\frac{2n-1}{4}c\Big).
\end{equation*}
This also yields  $nc=0$.

Summarizing the above discussion,  we thus assert the following:
\begin{proposition}\label{P2}
A hypersurface with $A\xi=\alpha\xi,\alpha\neq0$ in  $\widetilde{M}^n(c)$ does not admit a quasi-Einstein metric.
\end{proposition}
Together Proposition \ref{P1} with Proposition \ref{P2}, we complete the proof of Theorem \ref{T1}.

\section{Proof of Theorem \ref{T2}}
In this section we study a class of non-Hopf hypersurfaces with quasi-Einstein metric of non-flat complex space forms.
Let $\gamma:I\rightarrow\widetilde{M}^n(c)$ be any regular curve. For $t\in I$, let $\widetilde{M}^n_{(t)}(c)$ be a totally geodesic complex hypersurface
through the point $\gamma(t)$ which is orthogonal to the holomorphic plane spanned by $\gamma'(t)$ and $J\gamma'(t)$.
Write $M=\{\widetilde{M}^n_{(t)}(c):t\in I\}$. Such a construction asserts that $M$ is a real hypersurface of $\widetilde{M}^n(c)$, which is called a \emph{ruled hypersurface}. It is well-known that the shape operator $A$ of $M$ is written as:
\begin{equation}\label{4.1}
\begin{aligned}
A\xi=&\alpha\xi+\beta W\;(\beta\neq0), \\
AW = &\beta\xi,\\
AZ = &0\; \text{for any}\; Z\bot\xi,W,
\end{aligned}
\end{equation}
where $W$ is a unit vector field orthogonal to $\xi$, and $\alpha,\beta$ are differentiable functions on $M$. From \eqref{9}, we have
\begin{align}
Q\xi =& (\frac{1}{2}(n-1)c-\beta^2 )\xi,\label{34}\\
QW =& (\frac{1}{4}(2n+1)c-\beta^2 )W,\label{35}\\
QZ =& (\frac{1}{4}(2n+1)c)Z\quad \hbox{for any}\quad Z\bot\xi,W.\label{36}
\end{align}
From these equations we know that the scalar curvature $S=(n^2-1)c-2\beta^2$.

First we assume $n\geq3$ and write
\begin{equation*}
  T_1M=\{X\in TM:\eta(X)=g(X,W)=g(X,\phi W)=0\}.
\end{equation*}
We know that the following relations are valid (see \cite[Eq.(18),(15)]{K}):
\begin{equation*}
  \phi W(\beta)=\beta^2+c/4\quad\hbox{and}\quad X(\beta)=0\quad\hbox{for all}\;X\in T_1(M).
\end{equation*}
On the other hand, the Codazzi
equation \eqref{8} implies that $(\nabla_\xi A)W-(\nabla_WA)\xi = \frac{c}{4}\phi W,$ and using \eqref{4.1} we get
\begin{align*}
  (\nabla_\xi A)W-(\nabla_WA)\xi=&\nabla_\xi(AW)-A\nabla_\xi W-\nabla_W(A\xi)+A\nabla_W\xi\\
  =&(\xi(\beta)-W(\alpha))\xi + \beta^2\phi W- A\nabla_\xi W-W(\beta)W-\beta\nabla_WW,
\end{align*}
which, by taking an inner product with $W$, yields $W(\beta)=0$. Thus we have
\begin{equation}\label{4.4}
\nabla\beta=(\beta^2+c/4)\phi W+\xi(\beta)\xi.
\end{equation}

Furthermore, the following lemma holds:
\begin{lemma}[\cite{K}]\label{L5}
For all $Z\in T_1M$, we have the following relations:
\begin{align*}
\nabla_W\phi W &= \Big(\frac{c}{4\beta}-\beta\Big) W, \quad\nabla_WW =(\beta-\frac{c}{4\beta})\phi W,\\
\nabla_Z\phi W &= \frac{c}{4\beta} Z, \quad\nabla_ZW =-\frac{c}{4\beta}\phi Z,\\
\nabla_{\phi W}W&=0,\quad\nabla_{\phi W} \phi W = 0.
\end{align*}
\end{lemma}
For $Z\in T_1M$, from \eqref{4.4} we know $Z(\beta)=0$,
Putting $Y=\xi$ and $X=Z$ in \eqref{15}, we have
\begin{align}\label{40}
    &(\nabla_\xi Q)Z-(\nabla_ZQ)\xi+\frac{1}{m}\{Z(f)Q\xi-\xi(f)QZ\} \\
 = &\Big(\frac{c}{4}-\frac{\lambda}{m}\Big)\Big(\xi(f) Z-Z(f)\xi\Big)+A\xi(f)AZ-AZ(f)A\xi.\nonumber
\end{align}
Since $Z(\beta)=0$,  we obtain
\begin{equation*}
(\nabla_ZQ)\xi-(\nabla_\xi Q)Z=-\frac{c}{4}(2n+1)\nabla_\xi Z+Q\nabla_\xi Z.
\end{equation*}

By \eqref{34} and \eqref{36}, the inner product of \eqref{40} with $\xi$ gives
\begin{align*}
 Z(f)\Big[\frac{1}{m}\Big(\frac{1}{2}(n-1)c-\beta^2\Big)+\frac{c}{4}-\frac{\lambda}{m}\Big]=0.
\end{align*}
Similarly, putting $X=Z$ and $Y=W$ in \eqref{15}, we obtain
\begin{align*}
 Z(f)\Big[\frac{1}{m}\Big(\frac{1}{4}(2n+1)c-\beta^2\Big)+\frac{c}{4}-\frac{\lambda}{m}\Big]=0.
\end{align*}
The previous two formulas give
\begin{equation*}
 Z(f)=0.
\end{equation*}

Now putting $Y=\xi$ and $X=W$ in \eqref{15} yields
\begin{equation}\label{4.7}
\left\{
  \begin{array}{ll}
    W(f)\Big[\frac{1}{m}\Big(\frac{1}{2}(n-1)c-\beta^2\Big)+\frac{c}{4}-\frac{\lambda}{m}-\beta^2\Big] &=0,  \\
    \xi(f)\Big[\frac{c}{4}-\frac{\lambda}{m}+\frac{1}{m}\Big(\frac{1}{4}(2n+1)c-\beta^2\Big )-\beta^2\Big]&= -\xi(\beta^2).
  \end{array}
\right.
\end{equation}
 Here we have used \eqref{4.4} and $g(\nabla_\xi W,W)=g(\nabla_\xi W,\xi)=0$.

{\bf Case I:}$\frac{1}{m}\Big(\frac{1}{2}(n-1)c-\beta^2\Big)+\frac{c}{4}-\frac{\lambda}{m}-\beta^2=0$. Then $\beta$ is constant and $\beta^2=-\frac{c}{4}$ by \eqref{4.4}. Then
\begin{equation}\label{4.8}
\lambda=\frac{1}{4}(2n-1+2m)c .
\end{equation}
Moreover, from \eqref{4.7} we have
\begin{equation*}
\xi(f)=0.
\end{equation*}
Thus we may write
\begin{equation*}
  \nabla f=W(f)W+\phi W(f)\phi W.
\end{equation*}

For $m\neq1$, since $S=(n^2-\frac{1}{2})c$ is constant, it follows from \eqref{2.13} that
\begin{align*}
  &W(f)[\frac{1}{4}(2n+1)c-\beta^2]W+\phi W(f)[\frac{1}{4}(2n+1)c]\phi W\\
=&-\frac{1}{m-1}(S-(2n-2)\lambda)(W(f)W+\phi W(f)\phi W).
\end{align*}
By the orthogonality of $\phi W$ and $W$, we obtain
\begin{equation*}
\left\{
  \begin{array}{ll}
    W(f)\Big[\frac{1}{4}(2n+1)c-\beta^2+\frac{1}{m-1}\Big(S-(2n-2)\lambda\Big)\Big] &=0,  \\
   \phi W(f)\Big[\frac{1}{4}(2n+1)c+\frac{1}{m-1}\Big(S-(2n-2)\lambda\Big)\Big]&= 0.
  \end{array}
\right.
\end{equation*}
Because $m>1$, by \eqref{4.8} a direct computation implies
\begin{equation*}
  W(f)=\phi W(f)=0.
\end{equation*}

For $m=1$, it follows from \eqref{2.13} that $\nabla f=0$ or $S=(2n-2)\lambda$, i.e.
\begin{equation*}
(n^2-\frac{1}{2})c=\frac{1}{2}(n-1)(2n+1)c.
\end{equation*}
This is impossible since $M$ does not be an Einstein hypersurface as in introduction.

{\bf Case II:}$\frac{1}{m}\Big(\frac{1}{2}(n-1)c-\beta^2\Big)+\frac{c}{4}-\frac{\lambda}{m}-\beta^2\neq0$. Thus $W(f)=0$ by \eqref{4.7}.
Now letting $X=\xi$ and $Y=\phi W$ in \eqref{15} gives
\begin{equation}\label{4.8}
  \xi(f)\Big[\frac{c}{4}-\frac{\lambda}{m}+\frac{1}{4m}(2n+1)c\Big]=0
\end{equation}
and
\begin{align}\label{4.11*}
& \phi W(f)\Big[\frac{1}{m}\Big(\frac{1}{2}(n-1)c-\beta^2\Big)+\frac{c}{4}-\frac{\lambda}{m}\Big]\\
 &+\phi W(\beta^2)-\frac{1}{4}(2n+1)c\beta+\Big(\frac{1}{2}(n-1)c-\beta^2\Big)\beta=0.\nonumber
\end{align}
Meanwhile, taking $X=\phi W$ and $Y=W$ in \eqref{15} and applying Lemma \ref{L5}, we obtain
\begin{align}\label{4.12*}
 \phi W(f)\Big[\frac{1}{m}\Big(\frac{1}{4}(2n+1)c-\beta^2\Big)+c-\frac{\lambda}{m}\Big]+\beta^2\Big(\frac{c}{4\beta}-\beta\Big)+\phi W(\beta^2)=0.
\end{align}
Comparing \eqref{4.11*} with \eqref{4.12*} gives
\begin{align*}
 \phi W(f)(3m+3)+4m\beta=0.
\end{align*}

On the other hand, by using \eqref{4.8}, we follow from Equation \eqref{4.7} that
\begin{equation*}
\xi(f)=\frac{\xi(\beta^2)}{1+\frac{1}{m}}.
\end{equation*}
This means that
\begin{equation*}
  \nabla f=-\frac{4m\beta}{3(m+1)}\phi W+\frac{\xi(\beta^2)}{1+\frac{1}{m}}\xi,
\end{equation*}
hence for any $X,Y\in TM$,
\begin{align*}
  {\rm Hess} f(X,Y)=&g(\nabla_X\nabla f,Y)\\
  =&-\frac{4m}{3(m+1)}\Big[X(\beta)g(Y,\phi W)+\beta g(\nabla_X\phi W,Y)\Big]\\
  &+\frac{m}{m+1}\Big[X(\xi(\beta^2))\eta(Y)+\xi(\beta^2) g(\phi AX,Y)\Big].
\end{align*}
By Lemma \ref{L5} and \eqref{4.4}, we compute
\begin{equation}\label{4.11}
\left\{
  \begin{array}{ll}
     {\rm Hess} f(W,W)&=-\frac{4m}{3(m+1)}\Big(\frac{c}{4}-\beta^2\Big),\\
   {\rm Hess} f(\phi W,\phi W)&=-\frac{4m}{3(m+1)}\Big(\beta^2+\frac{c}{4}\Big).
  \end{array}
\right.
\end{equation}

On the other hand, using \eqref{35} and \eqref{36}, it follows from Equation \eqref{1} that
\begin{equation}\label{4.12}
\left\{
  \begin{array}{ll}
    {\rm Hess} f(W,W)&=\lambda-\Big(\frac{1}{4}(2n+1)c-\beta^2 \Big),\\
   {\rm Hess} f(\phi W,\phi W)&=\lambda+\frac{1}{m}\frac{(4m\beta)^2}{[3(m+1)]^2}-\frac{1}{4}(2n+1)c.
  \end{array}
\right.
\end{equation}
Combining \eqref{4.11} with \eqref{4.12}, we obtain

\begin{equation*}
15m^2+22m-9=0.
\end{equation*}
This equation has no solution for $m\geq1$.

For the case $n=2$, it is obvious that these equations including from \eqref{4.7} to \eqref{4.12} still hold,  we thus complete the proof of Theorem \ref{T2}.

\section{Proofs of Theorem \ref{T3} and Corollary \ref{T4}}

In this section we assume $c=0$. Namely $\widetilde{M}^n$ is a complex Euclidean space $\mathbb{C}^n$.

{\it Proof of Theorem \ref{T3}.} For a contact hypersurface, by \cite[Lemma 3.1]{CK}, we know that $M$ is Hopf and $\alpha=\eta(A\xi)$ is constant. Therefore we find that Equation \eqref{29} holds and can be simply as
\begin{align}\label{3.11*}
  &\alpha\mu^2(\mu-\alpha)^2=0.
\end{align}
This shows that $\mu$ is also constant, and further the scalar curvature is constant.
For $c=0$, Equation \eqref{3.10} becomes
\begin{align}\label{3.11}
  &\frac{\lambda-h\alpha+\alpha^2}{m}\{\nabla f-\xi(f)\xi\}+\alpha\Big(\alpha\xi(f)\xi-A\nabla f\Big)=0.
\end{align}
Taking an inner product of \eqref{3.11} with $X\in\mathfrak{D}$, then
\begin{align}\label{5.3}
  &\Big(\frac{\lambda-h\alpha+\alpha^2}{m}-\alpha \mu\Big)X(f)=0.
\end{align}
Next we decompose two cases.

{\bf Case I:} $\lambda\neq m\alpha\mu+h\alpha-\alpha^2$. We find $\nabla f=\xi(f)\xi$ by \eqref{5.3}. Then $M$ is a sphere as the proof of \cite[Theorem 3.2]{CK}.

{\bf Case II:} $\lambda=m\alpha\mu+h\alpha-\alpha^2$. If $\alpha=0$ then $\mu\neq0$, otherwise $M$ is totally geodesic, which is impossible.  In this case $M$ is a generalized cylinder $\mathbb{R}^{n}\times\mathbb{S}^{n-1}$. Next we assume $\alpha\neq0$, then $\mu=0$ or $\mu=\alpha$ by \eqref{3.11*}. If $\mu=0$, $M$ is $\mathbb{R}^{2n-1}$, which fails to be a contact hypersurface. Thus $\mu=\alpha$, $M$ is a totally
umbilical hypersurface.  Consequently it is a portion of a $(2n-1)$-dimensional sphere. Moreover, since $\lambda=(m+2n-2)\alpha^2>0$, $M$ is compact (see \cite{Q}).\qed\bigskip

{\it Proof of Corollary \ref{T4}.} If $A\xi=0$,  Formula \eqref{3.11} becomes
\begin{equation*}
 \frac{\lambda}{m}\Big(\xi(f)\xi-\nabla f\Big)=0.
\end{equation*}

When $\lambda\neq0$, we have $\nabla f=\xi(f)\xi$. Thus Equation \eqref{3.7*} holds.
By \eqref{11}, $A\phi A=0$ then we get $\xi(f)A^2\phi Z=0$. Moreover, $\xi(f)A^2Z=0$ for any $Z\in TM$ then either $A=0$ or $\xi(f)=0$. If $A=0$ then $Q=0$, then from \eqref{3.10*} we find $\lambda=0$, which is a contradictory to the assumption.
Thus $\xi(f)=0$, i.e. $f$ is constant. That means that $M$ is Einstein and the scalar curvature $S=(2n-1)\lambda$ by quasi-Einstein equation \eqref{1}.
We complete the proof by \cite[Theorem 7.3]{F}.\qed


\begin{thebibliography}{99}

\bibitem{B} J. Berndt, {\it Real hypersurfaces with constant principal curvatures in complex hyperbolic space,} J. Reine Angew. Math. {\bf 395} (1989), 132-141.
 \bibitem{CSW}J. Case, Y. J. Shu, G. Wei, {\it Rigidity of Quasi-Einstein Metrics,} Diff. Geom. and its Appl. {\bf29} (2011), 93-100.
\bibitem{C} J. Case, {\it On the non-existence of quasi-Einstein metrics}, Pac. J. Math. {\bf248}(2) (2010), 227-284.
\bibitem{CR} T. E. Cecil, P. J. Ryan, {\it Focal set and real hypersurfaces in complex projective spaces,} Trans. Amer. Math. Soc. {\bf 269} (1982), 481-499.
\bibitem{CK}J. T. Cho, M. Kimura, {\it Ricci solitons of compact real hypersurfaces in K\"ahler manifolds,}
             Math. Nachr. {\bf 284} (2011), 1385-1393.
\bibitem{CK2} J. T. Cho, M. Kimura, {\it Ricci solitons and real hypersurfaces in a complex space form,}
             Tohoku Math. J. {\bf61} (2009), 205-212.
\bibitem{F} A. Fialkow, {\it Hypersurfaces of spaces of constant curvature,} Ann. of Math. {\bf39} (1938), 762-785.
 \bibitem{G}A. Ghosh, {\it Quasi-Einstein contact metric manifolds}, Glasgow Math. J. {\bf57} (2015), 569-577.




\bibitem{K} M. Kimura, {\it Sectional curvatures of holomorphic planes on a real hypersurfaces in $P^n(\mathbb{C})$}, Math. Ann. {\bf276} (1987), 487-497.

\bibitem{KS} U.-H. Ki, Y. J. Suh, {\it On real hypersurfaces of a complex space form}, Math. J. Okayama {\bf32} (1990), 207-221.







\bibitem{M} S. Montoel, {\it Real hypersurfaces of a complex hyperbolic space},  J. Math. Soc. Japan {\bf 35} (1985), 515-535.



\bibitem{O}M. Okumura, {\it Contact hypersurfaces in certain Kaehlerian manifolds}, T\^{o}hoku Math. J. {\bf18} (1966), 74-102.
\bibitem{NR}R. Niebergall, P. J. Ryan, {\it Real hypersurfaces in complex space forms}, Tight and taut submanifolds
(eds. T. E. Cecil and S. S. Chern), Math. Sci. Res. Inst. Publ. {\bf 32} (1997), Cambridge Univ. Press, 233-305.
\bibitem{P}G. Perelman, {\it The entropy formula for the Ricci flow and its geometric applications}, http://arXiv.org/
abs/math.DG/02111159, preprint.

\bibitem{Q} Z. Qian, {\it Estimates for weighted volumes and applications}, Quart. J. Math. Oxford Ser. {\bf48}(2) (1997), 235-242.






\bibitem{V}Micheal H. Vernon, {\it Contact hypersurfaces of a complex hyperbolic space}, T\^{o}hoku Math. J.
{\bf39}(2) (1987),  215-222.
\bibitem{W}L. F. Wang,  {\it Rigid properties of quasi-Einstein metrics,} Proc. Amer. Math. Soc. {\bf139} (2011), 3679-3689.
\bibitem{W2}L. F. Wang, {\it Gap results for compact quasi-Einstein metrics}, Sci. China Math.  {\bf61}(5) (2018), 943-954.
\end{thebibliography}
\end{document}